\documentclass[10pt,reqno]{amsart}
\usepackage{amsmath,amssymb,amsfonts,amsthm,mathrsfs,enumitem}
\usepackage{url}
\allowdisplaybreaks

\newtheorem{thm}{Theorem}[section]

\newtheorem{cor}[thm]{Corollary}
\newtheorem{lem}[thm]{Lemma}
\newtheorem{prop}[thm]{Proposition}

\theoremstyle{definition}

\theoremstyle{remark}
\newtheorem{rem}[thm]{Remark}

\numberwithin{equation}{section}

\newcommand{\F}{\mathbb F}
\newcommand{\Tr}{\operatorname{Tr}}

\newcommand{\gauss}[2]{\genfrac{[}{]}{0pt}{}{#1}{#2}_{\!2}}
\newcommand{\Kl}{\operatorname{Kl}}

\begin{document}

\title[Counting subspaces with prescribed Welch sums]
{Counting subspaces with prescribed Welch sums}
\author{Kaimin Cheng}
\address{School of Mathematical Sciences, China West Normal University, Nanchong 637002, P. R. China}
\email{ckm20@126.com}

\author{Jie Hu}
\address{School of Mathematical Sciences, China West Normal University, Nanchong 637002, P. R. China}
\email{2629664270@qq.com}
\date{}

\begin{abstract}
Let $q=2^n$ with $n=2m+1$ for an integer $m\ge1$, and let $\F_q$ be the finite field with $q$ elements. Let $W_n(x)=x^{2^m+3}$ be the Welch function on $\F_q$. For
$1\le k\le n$ and $b\in\F_q$, denote by $N_{n,k}(b)$ the number of
$k$-dimensional $\F_2$-subspaces $E$ of $\F_q$ such that
\[
 \sum_{x\in E}W_n(x)=b.
\]
In this paper, we determine $N_{n,k}(b)$ for all $k$. The only nontrivial case is $k=3$,
for which the answer is given in terms of a binary Kloosterman sum. As a
consequence, we prove a recent conjecture of Hou and Zhao on the third-order
sum-freedom of the Welch function.
\end{abstract}

\maketitle

\section{Introduction}

Almost perfect nonlinear (APN) functions play an important role in symmetric
cryptography. They were introduced in connection with differential
cryptanalysis. In characteristic two, an APN function has differential
uniformity two, which is the smallest possible value; see, for example,
\cite{Nyberg1994}. Power functions over binary fields give many important
examples. One of them is the Welch function, which is defined by
\[
W_n(x)=x^{2^m+3},\qquad n=2m+1.
\]
Dobbertin \cite{Dobbertin1999} proved that the Welch function is APN. Its
almost bent property, or equivalently the three-valued cross-correlation
property conjectured by Welch, was proved by Canteaut, Charpin and Dobbertin
\cite{CanteautCharpinDobbertin2000}.

Carlet \cite{Carlet2025} introduced \emph{$k$th-order sum-freedom}. A function
$f:\F_{2^n}\to\F_{2^n}$ is $k$th-order sum-free if
\[
 \sum_{x\in A}f(x)\ne0
\]
for every $k$-dimensional affine $\F_2$-subspace $A$ of $\F_{2^n}$. When
$k=2$, this is equivalent to APNness. More generally, if
$A=z+\langle e_1,\ldots,e_k\rangle$, then
\[
 \sum_{x\in A}f(x)=D_{e_1}\cdots D_{e_k}f(z),
\]
where $D_e f(x)=f(x+e)+f(x)$. Thus sum-freedom is related to higher-order
derivatives and to integral cryptanalysis; see \cite{Carlet2025}.

The multiplicative inverse function was studied first in this setting.
Several results on its sum-freedom were obtained in
\cite{CarletInverse2025,CarletHou2025,EbelingHouRydellZhao2026}, and Carlet's
conjecture for the inverse function was recently proved in
\cite{Cheng2026Inverse}. Hou and Zhao \cite{HouZhao2026} also considered
$q$-ary functions and classical power functions, including the Welch
function.

For the Welch function, only dimension three remains to be considered. When
$n\ge5$, the binary weight of $2^m+3=2^m+2+1$ is three, so $W_n$ has algebraic
degree three. Since $W_n$ is APN, it is second-order sum-free. On the other
hand, every derivative of order at least four is zero. Hence the sum of
$W_n$ over every affine subspace of dimension at least four is zero. Hou and
Zhao reduced the third-order case to a rank condition for a Dickson matrix.
They verified by computation that the Welch function is not third-order
sum-free for $7\le n\le15$, and conjectured that the same holds for every odd
$n\ge7$ \cite[Conjecture~5.10]{HouZhao2026}.

Throughout the paper, let
\[
n=2m+1,\qquad q=2^n,\qquad F=\F_q,
\]
let $\Tr=\Tr_{F/\F_2}$ be the absolute trace, and denote by $E\le F$ the $\F_2$-subspace $E$ of $F$. For $1\le k\le n$ and
$b\in F$, define
\[
 N_{n,k}(b)=\#\left\{E\le F:\dim_{\F_2}E=k,
 \ \sum_{x\in E}W_n(x)=b\right\}.
\]
We also use the binary Kloosterman sum
\begin{equation}\label{eq:Kloosterman}
 \Kl_n=\sum_{u\in F^*}(-1)^{\Tr(u+u^{-1})}.
\end{equation}
Here $\gauss nk$ denotes the Gaussian binomial coefficient, which counts the number of $k$-dimensional $\F_2$-subspaces of $F$.

\begin{thm}\label{thm:main}
For every odd $n\ge3$, we have
\[
 N_{n,k}(0)=
 \begin{cases}
  0, & k=1,2,\\
  \dfrac{(q-1)(q+1-3\Kl_n)}{168}, & k=3,\\
  \gauss nk, & 4\le k\le n,
 \end{cases}
\]
and, for every $b\ne0$,
\[
 N_{n,k}(b)=
 \begin{cases}
  1, & k=1,\\
  \dfrac{q-2}{6}, & k=2,\\
  \dfrac{q^2-7q+7+3\Kl_n}{168}, & k=3,\\
  0, & 4\le k\le n.
 \end{cases}
\]
Moreover, if $B_n$ denotes the number of three-dimensional zero-Welch-sum subspaces
of $F$ containing $1$, then
\begin{equation}\label{eq:B}
 B_n=\frac{q+1-3\Kl_n}{24}.
\end{equation}
\end{thm}

Clearly, Theorem~\ref{thm:main} gives
\[
 N_{n,3}(0)\sim 2^{-n}\gauss n3\qquad(n\to\infty,\ n\text{ odd}).
\]
Thus the zero value has the asymptotic frequency $2^{-n}$. We also obtain the
following consequence.

\begin{cor}\label{cor:HZ}
For every odd $n\ge7$, the Welch function on $\F_{2^n}$ is not third-order
sum-free.
\end{cor}

For Theorem~\ref{thm:main}, the cases $k=1,2$ and $k\ge4$ follow immediately from the permutation property of the Welch exponent, the APN property of $W_n$, and the fact that $W_n$ has algebraic degree three, respectively. Thus it remains to consider the case $k=3$. For this case, we use the map $\phi(x)=x^2+x$ and a partial trace map to reduce the problem to the study of certain two-dimensional subspaces of the trace-zero hyperplane; see Section~2. This reduction leads to a criterion involving $\Tr(u^{-1})$. A counting argument based on this criterion then gives the formulas in Theorem~\ref{thm:main}; see Section~3.  Finally, in
Section~4, we relate the Kloosterman sum $\Kl_n$ to a fixed elliptic curve,
derive a recurrence for $\Kl_n$, and recover the Dickson determinant
computations of Hou and Zhao.

\section{A partial trace map on the trace-zero hyperplane}

First of all, we put
\[
 T=\ker\Tr,
 \qquad
 S(t)=\sum_{j=0}^{m-1}t^{2^j},
 \qquad
 r=2^{m+1},
 \qquad
 G(y)=y^2+y^r.
\]
Since $n=2m+1$, one has
\begin{equation}\label{eq:rproperties}
z^{r^2}=z^2\quad(z\in F),\ \gcd(m+1,n)=1,\ (r-1)(r+1)\equiv1\pmod{q-1}.
\end{equation}
Hence the power maps $z\mapsto z^{r-1}$ and $z\mapsto z^{r+1}$ are inverse permutations of $F^*$.

\begin{lem}\label{lem:inverse}
The restrictions of $S$ and $G$ to $T$ are mutually inverse
$\F_2$-linear automorphisms of $T$.
\end{lem}

\begin{proof}
Let $t, y\in T$. One has 
\[
\Tr(S(t))=\Tr\left(\sum_{i=0}^{m-1}t^{2^i}\right)=\sum_{i=0}^{m-1}\Tr(t^{2^i})=m\Tr(t)=0,
\]
and
$$\Tr(G(y))=\Tr(y^2)+\Tr(y^r)=0.$$
It implies that both $S(t)$ and $G(y)$ are in $T$. Thus, both maps preserve $T$. Furthermore,
for every $t\in F$,
\[
 G(S(t))
=S(t)^2+S(t)^r
=\sum_{j=1}^{m}t^{2^j}+\sum_{j=m+1}^{2m}t^{2^j}
 =t+\Tr(t).
\]
It follows that $G\circ S$ is the identity on $T$. In addition, both maps are endomorphisms of
the finite vector space $T$, so they are inverse automorphisms.
\end{proof}

Now we define an $\F_2$-linear map $\phi$ from $F$ to itself by $\phi(x)=x^2+x$. Clearly,
$\ker\phi=\F_2=\langle1\rangle$ and $\operatorname{im}\phi=T$.

\begin{lem}\label{lem:phi-bijection}
The map
\[
H\longmapsto\phi(H)
\]
is a bijection from the three-dimensional subspaces of $F$ containing $1$
to the two-dimensional subspaces of $T$, where $\phi(H)$ denotes the image of the restriction of $\phi$ to $H$.
\end{lem}

\begin{proof}
Let $H$ be a three-dimensional subspace of $F$ containing $1$. Since
$\ker\phi=\langle1\rangle\subset H$, the rank-nullity theorem gives
\[
\dim_{\F_2}\phi(H)
=
\dim_{\F_2}H-\dim_{\F_2}\ker(\phi|_H)
=
3-1=2.
\]
Moreover, $\phi(H)\subset T$.

Conversely, let $V$ be a two-dimensional subspace of $T$. Since
$\operatorname{im}\phi=T$, the preimage
$H=\phi^{-1}(V)$
is a subspace of $F$ containing $\ker\phi=\langle1\rangle$, and
\[
\dim_{\F_2}H
=
\dim_{\F_2}V+\dim_{\F_2}\ker\phi
=
2+1=3.
\]
Thus $H$ is a three-dimensional subspace of $F$ containing $1$, and
$\phi(H)=V$. These two constructions are inverse to each other. This proves Lemma \ref{lem:phi-bijection}.
\end{proof}

Let $H$ be a three-dimensional subspace of $F$ containing $1$. We write $H=\langle1,x,y\rangle_{\F_2}$ and put
$t=x^2+x$, $u=y^2+y$. By a direct expansion, one obtains
\begin{align}
 \sum_{z\in H}W_n(z)
 = (x^2+x)(y^{2^m}+y)+(y^2+y)(x^{2^m}+x)=tS(u)+uS(t),\label{eq:polarization}
\end{align}
where we used $S(x^2+x)=x^{2^m}+x$.

For $a\in F$, define the binary linear subspace
\begin{equation*}
 V_a=\{y\in F:G(y)=ay\}
     =\{y\in F:y^2+y^r+ay=0\}.
\end{equation*}

\begin{lem}\label{lem:dimVa}
For every $a\in F$, one has $\dim_{\F_2}V_a\le2$.  Moreover, $V_0=\F_2$.
\end{lem}

\begin{proof}
If $y\in V_a$, then $y^r=y^2+ay$.  Applying the $r$th power and using
\eqref{eq:rproperties}, one obtains
\[
 y^4+(a^2+a^r+1)y^2+a^{r+1}y=0.
\]
Therefore $V_a$ has at most four elements. It implies $\dim_{\F_2}V_a\le2$. When $a=0$, the displayed equation
reduces to $y^4+y^2=0$, and hence $V_0=\F_2$.
\end{proof}

\begin{prop}\label{prop:B-kernels}
Let $B_n$ be the number of three-dimensional $\F_2$-subspaces of $F$ containing $1$, with zero Welch sum. Then $B_n$ equals the number of $a\in F^*$ for which
$V_a$ has dimension two and is contained in $T$.
\end{prop}

\begin{proof}
Let $H$ be a three-dimensional subspace of $F$ containing $1$, with zero Welch sum. Let $V=\phi(H)$. By Lemma \ref{lem:phi-bijection}, we know that $V$ is a two-dimensional subspace of $T$. Now we write $V=\langle t,u\rangle_{\F_2}$. Note that $H$ has zero Welch sum. It follows from \eqref{eq:polarization} that
\[
\frac{S(t)}{t}=\frac{S(u)}{u}=s
\]
for some $s\in F^*$. By the linearity of $S$, we then have
$S(v)=sv$
for every $v\in V$. The scalar $s$ is uniquely determined by $V$, since
$V$ contains a nonzero element. Now put $U=S(V)$. Since $V\subset T$ and $S$ is an automorphism of $T$, the
space $U$ is a two-dimensional subspace of $T$. For every $y\in U$, there
is some $y'\in V$ such that
$y=S(y')=sy'$.
Let $a=s^{-1}$. By Lemma~\ref{lem:inverse}, we deduce that
\[
G(y)=G(S(y'))=y'=ay.
\]
Hence $U\subseteq V_a$. Since $U$ has dimension two,
Lemma~\ref{lem:dimVa} gives $U=V_a$. Thus $H$ determines a unique
$a\in F^*$ such that $\dim_{\F_2}V_a=2$ and $V_a\subset T$.

Conversely, suppose that $a\ne0$, $\dim_{\F_2}V_a=2$, and $V_a\subset T$. Put $V=G(V_a)$.
Since $G(y)=ay$ for every $y\in V_a$, we have
 $V=aV_a$.
Hence $V$ is a two-dimensional subspace of $T$. Let $v\in V$. Then $v=G(y)$ for some $y\in V_a$, and therefore
$v=ay$. By Lemma~\ref{lem:inverse},
\[
S(v)=S(G(y))=y=a^{-1}v.
\]
Thus $S(v)=a^{-1}v$for every $v\in V$. Now choose a basis $V=\langle t,u\rangle_{\F_2}$. Since
$H=\phi^{-1}(V)$, choose $x,y\in H$ such that
$\phi(x)=t$ and $\phi(y)=u$.
Then $H=\langle1,x,y\rangle_{\F_2}$. By
Lemma~\ref{lem:phi-bijection}, the preimage $H=\phi^{-1}(V)$
is a three-dimensional subspace of $F$ containing $1$. Moreover, by
\eqref{eq:polarization}, we have
\[
\sum_{x\in H}W_n(x)
=tS(u)+uS(t)
=a^{-1}tu+a^{-1}ut
=0.
\]
Hence $H$ has zero Welch sum.

Therefore, the two constructions are inverse to each other. This proves Proposition \ref{prop:B-kernels}.
\end{proof}

\section{The inverse-trace criterion and counting}

We shall use the following lemma.

\begin{lem}\label{lem:tracecriterion}
Let $a\in F^*$ and $u\in V_a\setminus\{0\}$. Then we have
\begin{equation*}
\dim_{\F_2}V_a=2
 \quad\Longleftrightarrow\quad
 \Tr(u^{-1})=1.
\end{equation*}
\end{lem}

\begin{proof}
Let $P(z)=z^{r+1}+z$. By \eqref{eq:rproperties}, the map
\[
V_a\setminus\{0\}\longrightarrow \{z\in F:P(z)=a\},
\qquad
y\longmapsto y^{r-1},
\]
is a bijection, whose inverse is
$z\longmapsto z^{r+1}$. Indeed, for every $y\in V_a\setminus\{0\}$, $G(y)=y^2+y^r=ay$. It implies $y^{r-1}=a+y$. So, we have
$$P(y^{r-1})=(y^{r-1})^{r+1}+y^{r-1}=y+a+y=a.$$  Conversely, for $z\in F$ satisfying $P(z)=a\ne0$ and
$y=z^{r+1}$, one has $y^{r-1}=z$ and $a=z+z^{r+1}=y^{r-1}+y$. It implies that $G(y)=y^2+y^{r}=y(y+y^{r-1})=ay$, and then $y\in V_a$. In particular,
\begin{equation}\label{eq:zay}
 z=a+y.
\end{equation}

Assume first that $\dim_{\F_2}V_a=2$. Choose an element
$v\in V_a\setminus\{0,u\}$. By \eqref{eq:zay}, we can write $z_0$ and $z_0+u$ for the two roots of $P(z)=a$ corresponding to
$v$ and $v+u$ differ by $u$. Now let $w=z_0u^{-1}$.
Then from $P(z_0+u)=P(z_0)$, we obtain
\begin{equation}\label{eq:AS}
 w^r+w=1+u^{-r}.
\end{equation}
Taking traces in \eqref{eq:AS}, we get
\[
 0=\Tr(1+u^{-r})=1+\Tr(u^{-1}),
\]
where we used the invariance of the absolute trace under Frobenius powers
and the fact that $n$ is odd. Hence $\Tr(u^{-1})=1$.

Conversely, suppose that $\Tr(u^{-1})=1$. Consider the $\F_2$-linear map
\[
L(c)=c^r+c,\qquad c\in F.
\]
Since $\gcd(m+1,n)=1$, we have $\ker L=\F_2$. Moreover,
\[
\Tr(L(c))=\Tr(c^r)+\Tr(c)=0,
\]
so $\operatorname{im}L\subseteq T$. Since both $\operatorname{im}L$ and
$T$ have dimension $n-1$, it follows that
$\operatorname{im}L=T$.
On the other hand,
\[
\Tr(1+u^{-r})
=1+\Tr(u^{-1})
=0.
\]
Thus $1+u^{-r}\in T=\operatorname{im}L$, and hence there exists
$c_0\in F$ such that
\begin{equation}\label{eq:AS-converse}
c_0^r+c_0=1+u^{-r}.
\end{equation}

Now raising \eqref{eq:AS-converse} to the $r$th power and adding it to
\eqref{eq:AS-converse}, we obtain
\[
c_0^2+c_0=u^{-r}+u^{-2}.
\]
Also,
\[
c_0^r=c_0+1+u^{-r}.
\]
Therefore, we compute that
\begin{align*}
P(uc_0)
&=u^{r+1}c_0^{r+1}+uc_0\\
&=u^{r+1}c_0(c_0+1+u^{-r})+uc_0\\
&=u^{r+1}(c_0^2+c_0)+uc_0+uc_0\\
&=u^{r+1}(u^{-r}+u^{-2})\\
&=u+u^{r-1}.
\end{align*}
Since $u\in V_a\setminus\{0\}$, we have $a=u+u^{r-1}$.
Hence $P(uc_0)=a$. Then by \eqref{eq:AS-converse}, one deduces that
\begin{align*}
P(uc_0+u)+P(uc_0)
&=(uc_0+u)^{r+1}+uc_0+u+a\\
&=(uc_0)^{r+1}+uc_0+a\\
&=P(uc_0)+a\\
&=0.
\end{align*}
Thus $P(uc_0+u)=P(uc_0)=a$. Now
by \eqref{eq:zay}, these two roots correspond to
two nonzero elements $y_1=a+uc_0, y_2=a+uc_0+u$
of $V_a$. Since $y_1+y_2=u\ne0$,
$y_1$ and $y_2$ are linearly independent over
$\F_2$, and consequently
$\dim_{\F_2}V_a\ge2$. This together with Lemma~\ref{lem:dimVa}  concludes that $\dim_{\F_2}V_a=2$. Lemma \ref{lem:tracecriterion} is proved. 
\end{proof}

Now we focus on the counting. Let
\[
 \mathcal A=\{a\in F^*: \dim_{\F_2}V_a=2\},
 \qquad
 M_n=|\mathcal A|.
\]
For every $u\in F^*$ there is a unique parameter
\[
 a=u+u^{r-1}
\]
for which $u\in V_a$. Then we have $a=0$ if and only if
$u^{r-2}=1$. Note that
\[
\gcd(r-2,q-1)
=\gcd(2^m-1,2^{2m+1}-1)
=2^{\gcd(m,2m+1)}-1=1.
\]
Hence $a=0$ if and only if $u=1$.  Lemma~\ref{lem:tracecriterion} therefore shows that
\begin{align}\label{eq: partition}
\{u\in F^*: \Tr(u^{-1})=1\}\setminus\{1\}
 =
\bigsqcup_{a\in\mathcal A}\bigl(V_a\setminus\{0\}\bigr).
\end{align}
Since the map $ \Tr:F\to\mathbb F_2$
is surjective, we have $F/T\cong\mathbb F_2$. Hence
\[
\#T
 =
\#\{u\in F:\Tr(u)=1\}
 =
 \frac q2.
\]
Note that $1/x$ is a permutation of $F^*$. It follows that
\[
 \#\{u\in F^*:\Tr(u^{-1})=1\}
 =
 \frac q2.
\]  
Also $\Tr(1)=1$ since $n$ is odd. Then, by \eqref{eq: partition} one has
$$\frac{q}{2} -1=\#\{u\in F^*: \Tr(u^{-1})=1\}\setminus\{1\}
 =
\#\bigsqcup_{a\in\mathcal A}\bigl(V_a\setminus\{0\}\bigr)=\sum_{a\in\mathcal A}\#V_a\setminus\{0\}.$$
Since $\#V_a\setminus\{0\}=3$ for every $a\in\mathcal A$, one obtains
\begin{equation}\label{eq:M}
M_n=\frac{q/2-1}{3}=\frac{q-2}{6}.
\end{equation}

Next put
\[
 A_n=\#\{u\in F^*: \Tr(u)=0,\ \Tr(u^{-1})=1\}.
\]
We now write $\mathcal A=\mathcal A_1\sqcup\mathcal A_2$,
where $\mathcal A_1=\{a\in\mathcal A:V_a\subset T\}$ and 
$\mathcal A_2=\{a\in\mathcal A:V_a\not\subset T\}$.
If $a\in\mathcal A_1$, then all three nonzero elements of $V_a$
have trace zero. By Lemma~\ref{lem:tracecriterion}, they also satisfy
$\Tr(u^{-1})=1$, and hence all three are counted by $A_n$.

Now let $a\in\mathcal A_2$. Since $V_a\not\subset T$, the restriction $\Tr|_{V_a}:V_a\longrightarrow\mathbb F_2$
is nonzero, and hence is surjective. Therefore
\[
 \dim\ker(\Tr|_{V_a})
 =
 \dim V_a-\dim\mathbb F_2
 =
 2-1=1.
\]
Thus $\ker(\Tr|_{V_a})$ has two elements, one of which is $0$.
Consequently, there is exactly one nonzero element $u\in V_a$
such that $\Tr(u)=0$. By Lemma~\ref{lem:tracecriterion}, this element
also satisfies $\Tr(u^{-1})=1$. Hence each $V_a$ with
$a\in\mathcal A_2$ contributes exactly one element to $A_n$.

Let $B_n$ be the number of three-dimensional $\F_2$-subspaces of $F$ containing $1$, with zero Welch sum. By Proposition~\ref{prop:B-kernels},
\[
 \#\mathcal A_1=B_n,
 \qquad
 \#\mathcal A=M_n,
\]
and therefore $\#\mathcal A_2=M_n-B_n$. Hence
\begin{equation}\label{eq:ABM}
 A_n
 =
 3B_n+(M_n-B_n)
 =
 M_n+2B_n.
\end{equation}

Let $\chi(x)=(-1)^{\Tr(x)}$ be the canonical additive character of $F$.  By the definition of $A_n$, one has
\begin{align}
 A_n
&=\frac14\sum_{u\in F^*}
 (1+\chi(u))(1-\chi(u^{-1}))\notag\\
 &=\frac14\Big(
 q-1
+\sum_{u\in F^*}\chi(u)
 -\sum_{u\in F^*}\chi(u^{-1})
 -\sum_{u\in F^*}\chi(u+u^{-1})
\Big).\notag
\end{align}
Since inversion permutes $F^*$, the middle two sums are equal and hence
cancel. Therefore, by \eqref{eq:Kloosterman}, we have
\begin{equation}\label{eq:A}
 A_n=\frac{q-1-\Kl_n}{4}.
\end{equation}
Substituting \eqref{eq:M} and \eqref{eq:A} into \eqref{eq:ABM}, we find
\[
 B_n
 =\frac12\left(\frac{q-1-\Kl_n}{4}-\frac{q-2}{6}\right)
 =\frac{q+1-3\Kl_n}{24},
\]
which proves \eqref{eq:B}.

\begin{proof}[Proof of Theorem~\ref{thm:main}]
We first consider $k=3$. Let us count the pairs $(E,e)$ such that
\[ \dim_{\F_2}E=3,\qquad
\sum_{x\in E}W_n(x)=0,\qquad
 e\in E\setminus\{0\}.
\]
On the one hand, for each such subspace $E$, there are seven choices for $e$, and hence
the number of pairs is $7N_{n,3}(0)$.
On the other hand, fix $e\in F^*$. Multiplication by $e^{-1}$ gives a
bijection between the zero-Welch-sum three-dimensional subspaces
containing $e$ and those containing $1$. Hence there are exactly $B_n$
such subspaces containing each fixed $e\ne0$. This implies that the
same number of pairs is also equal to $(q-1)B_n$. Therefore
\[
 7N_{n,3}(0)=(q-1)B_n,
\]
which gives
\[
 N_{n,3}(0)=\frac{(q-1)(q+1-3\Kl_n)}{168}.
\]

We next count the subspaces with nonzero sum. Put $d=2^m+3$. First, one claims that 
$\gcd(d,q-1)=1$. In fact, let $g$ be a common divisor of $d$ and $q-1$. Since
$g\mid d$, we have
$2^m\equiv -3\pmod g$, and then $2^{2m+1}=2(2^m)^2\equiv18\pmod g$.
On the other hand, $g\mid q-1=2^{2m+1}-1$, that is, $2^{2m+1}\equiv1\pmod g$. Thus we have $g\mid17$. Now suppose $g=17$. Then $2^n\equiv1\pmod{17}$. Note that the multiplicative order of $2$ modulo $17$ is $8$. It implies
$8\mid n$, which is impossible because $n=2m+1$ is odd. Therefore
$g=1$, and hence
$\gcd(d,q-1)=1$. Thus
$x\mapsto x^d$ permutes $F^*$. Hence, for any $b,c\in F^*$, one can choose $\lambda\in F^*$
such that $\lambda^d=c/b$. Since
\[
\sum_{x\in\lambda E}W_n(x)
=\lambda^d\sum_{x\in E}W_n(x),
\]
the map $E\mapsto \lambda E$
gives a bijection between the subspaces with Welch sum $b$
and those with Welch sum $c$. Therefore all nonzero Welch sums occur
equally often. Consequently, for every $b\in F^*$
\[
 N_{n,3}(b)
 =
 \frac{\gauss n3-N_{n,3}(0)}{q-1}=\frac{q^2-7q+7+3\Kl_n}{168}.
\]

It remains to consider the other dimensions. If $k=1$, then
$E=\{0,a\}$ with $a\ne0$, and its sum is $a^{2^m+3}$. Since
$\gcd(2^m+3,q-1)=1$, every element $b$ of $F^*$ there exists a unique $a\in F^*$ such that $a^{2^m+3}=b$. Hence
$N_{n,1}(0)=0$ and $N_{n,1}(b)=1$ for $b\ne0$.

If $k=2$, the Welch function is APN \cite{Dobbertin1999}, hence it is
second-order sum-free. Thus $N_{n,2}(0)=0$. As above, all nonzero Welch sums occur
equally often. Then we have
$$N_{n,2}(b)=\frac{\gauss n2-N_{n,2}(0)}{q-1}=\frac{q-2}{6}$$
for every $b\in F^*$.

Finally, let $k\ge4$. Let $E=\langle e_1,\ldots,e_k\rangle_{\F_2}$ be a $k$ dimensional $\F_2$-subspace of $F$ with an $\F_2$- basis $(e_1,\ldots,e_k)$. Since $W_n$ has algebraic degree three,
every derivative of order at least four is identically zero. Hence
\[
\sum_{x\in E}W_n(x)=\sum_{\varepsilon_1,\ldots,\varepsilon_k\in\F_2}
W_n(\varepsilon_1e_1+\cdots+\varepsilon_ke_k)=D_{e_1}\cdots D_{e_k}W_n(0)
 =0,
\]
where $D_a W_n(x)=W_n(x+a)+W_n(x)$. 
Thus every $k$-dimensional subspace has zero Welch sum. It follows that $N_{n,k}(0)=\gauss nk$, and $N_{n,k}(b)=0$ for every $b\ne 0$. This completes the proof of Theorem \ref{thm:main}.
\end{proof}

\begin{proof}[Proof of Corollary~\ref{cor:HZ}]
By Theorem \ref{thm:main}, together with the classic Weil bound for binary Kloosterman sums $|\Kl_n|\le2\sqrt q$ (see, for example,
\cite[Theorem~6.8]{Moisio2008}), we have
\[
 N_{n,3}(0)\ge \frac{q-1}{168}(q+1-6\sqrt q)>0
 \qquad(q\ge128).
\]
This proves the assertion.
\end{proof}

\section{Kloosterman recurrence and the Dickson determinant}

In this section, we shall make the Kloosterman term appearing in
Theorem~\ref{thm:main} more explicit. For every integer $j\ge1$ we write
\[
\Kl_j=\sum_{u\in\F_{2^j}^*}
(-1)^{\Tr_{\F_{2^j}/\F_2}(u+u^{-1})}.
\]
We first realize $-\Kl_n$ as the
Frobenius trace of a fixed elliptic curve. This gives a simple recurrence for
the sequence $(\Kl_j)_{j\ge1}$, which will then be related to the Dickson
determinant arising in our counting formula. 

We recall some standard facts about elliptic curves over finite fields.
Let $Q$ be a prime power. Let $\mathcal E$ be an elliptic curve over $\F_Q$, and let
\[
 \pi:\mathcal E\longrightarrow\mathcal E,\qquad
 (x,y)\longmapsto(x^Q,y^Q)
\]
be the Frobenius endomorphism. We write
$\#\mathcal E(\F_Q)$ for the number of $\F_Q$-rational points on
$\mathcal E$, including the point at infinity. If
\[
 t=Q+1-\#\mathcal E(\F_Q),
\]
then $t$ is called the \textit{Frobenius trace}, and the characteristic
polynomial of the Frobenius endomorphism is
$ X^2-tX+Q$

We shall use the following standard fact, see \cite[Theorem V.2.3.1(a)]{Silverman2009}.

\begin{lem}\label{lem:frobenius-points}
Let $\alpha$ and $\beta$ be the roots of the characteristic polynomial
\[
 X^2-tX+Q
\]
of the Frobenius endomorphism of an elliptic curve $\mathcal E/\F_Q$.
Then, for every $j\ge1$,
\[
 \#\mathcal E(\F_{Q^j})
 =
 Q^j+1-(\alpha^j+\beta^j).
\]
\end{lem}

We now apply Lemma~\ref{lem:frobenius-points} to the elliptic curve
over $\F_2$
\[
 \mathcal E:\quad Y^2+XY=X^3+1.
\]
For $x\ne0$, put $z=y/x$. Dividing the equation by $x^2$, we obtain
\[
 z^2+z=x+x^{-2}.
\]
For $c\in F=\F_{2^n}$, the equation
$z^2+z=c$
has $1+\chi(c)$ solutions in $F$, where $\chi(c)=(-1)^{\Tr(c)}$.
Hence the number of affine points with $x\ne0$ is
\[
 \sum_{x\in F^*}\bigl(1+\chi(x+x^{-2})\bigr).
\]
Since
$\Tr(x^{-2})=\Tr(x^{-1})$ for any $x\in F^*$,
we have $\chi(x+x^{-2})=\chi(x+x^{-1})$.
Therefore, by \eqref{eq:Kloosterman}, the number of such points is $2^n-1+\Kl_n$. In addition, there is one affine point with $x=0$, namely $(0,1)$, and one point at
infinity. Thus
\begin{equation}\label{eq:curve}
 \#\mathcal E(\F_{2^n})
 =
 2^n+1+\Kl_n.
\end{equation}
Now taking $n=1$ in \eqref{eq:curve}, we obtain
$ \#\mathcal E(\F_2)=4$.
Hence the Frobenius trace of $\mathcal E/\F_2$ is $-1$
and the characteristic polynomial of its Frobenius endomorphism is
$X^2+X+2$. Let $\alpha=\frac{-1+i\sqrt7}{2}$ and $\beta=\frac{-1-i\sqrt7}{2}$
be its roots. By Lemma~\ref{lem:frobenius-points},
\[
 \#\mathcal E(\F_{2^n})
 =
 2^n+1-(\alpha^n+\beta^n).
\]
Comparing this with \eqref{eq:curve}, we obtain
\begin{equation}\label{eq:Kl-roots}
 \Kl_n=-(\alpha^n+\beta^n)\quad (n\ge 1).
\end{equation}
It implies that
\begin{equation}\label{eq:recurrence}
 \Kl_1=1,\qquad
 \Kl_2=3,\qquad
 \Kl_j=-\Kl_{j-1}-2\Kl_{j-2}
 \quad(j\ge3).
\end{equation}
Thus the quantities $\Kl_n$ appearing in Theorem~\ref{thm:main} can be computed
from this second-order recurrence.

Moreover, since
$|\alpha|=|\beta|=\sqrt2$, by \eqref{eq:Kl-roots} we recover the estimate used
in the previous section,
\[
 |\Kl_n|
 \le |\alpha|^n+|\beta|^n
 =2^{n/2+1}
 =2\sqrt q.
\]

We can also recover the determinant computations of Hou and Zhao.
Let $D^*(X)$ denote the matrix introduced in
\cite[Section~5.3]{HouZhao2026}.  For odd $n\ge5$, their criterion and
\eqref{eq:polarization} show that, apart from $x=0,1$, a root of
$\det D^*(X)$ in $F$ corresponds to a nonzero element
$t=x^2+x$ lying in a three-element fiber of $S(t)/t$.  By the proof of
Proposition~\ref{prop:B-kernels}, each zero-sum subspace containing $1$
corresponds to one such two-dimensional fiber, which has exactly three
nonzero elements.  Thus the $B_n$ zero-sum subspaces containing $1$ account for
$3B_n$ distinct values of $t$.  For each $t\in T^*$, the equation
$x^2+x=t$ has exactly two solutions, namely $x$ and $x+1$.  They give $6B_n$ roots outside $\F_2$; the two remaining roots
are $x=0,1$.  Since $X^q+X$ has simple roots, we obtain
\begin{equation}\label{eq:HZdegree}
\deg\gcd\bigl(\det D^*(X),X^q+X\bigr)
 =2+6B_n
 =\frac{q+9-3\Kl_n}{4}.
\end{equation}
Hence \eqref{eq:HZdegree} gives a closed formula for the numerical sequence
in \cite[Table~1]{HouZhao2026}.

\begin{rem}
For $n=5$, the recurrence \eqref{eq:recurrence} gives $\Kl_5=11$, and
Theorem~\ref{thm:main} yields $N_{5,3}(0)=0$. For $n=3$, the same theorem
gives $N_{3,3}(0)=1$. Corollary~\ref{cor:HZ} gives $N_{n,3}(0)>0$ for every odd
$n\ge7$.  It follows that the Welch function is third-order sum-free exactly for
$n=5$ among odd extension degrees $n\ge3$. For $n\ge5$, every three-dimensional affine
subspace with a fixed direction has the same Welch sum, so the corresponding
affine counts are simply $qN_{n,3}(b)/8$.
\end{rem}

\enlargethispage{2\baselineskip}

\end{document}